\documentclass[a4paper,12pt]{amsart}

\usepackage[utf8]{inputenc}
\usepackage{amsmath,amsfonts,amssymb,amsopn,amscd,amsthm}

\usepackage[left=3.5cm,top=2.5cm,bottom=2cm,right=3cm]{geometry}

\theoremstyle{plain}
\newtheorem{theorem}{Theorem}[section]

\newtheorem{corollary}[theorem]{Corollary}

\newtheorem{lemma}[theorem]{Lemma}

\theoremstyle{definition}
\newtheorem{example}{Example}

\title[Polyharmonic functions in cones]{Polyharmonic functions and random processes in cones}

\author{Fran\c cois Chapon}
\address{Universit\'e de Toulouse, Institut de Math\'ematiques de Toulouse, UMR CNRS 5219, UPS, F-31062 Toulouse Cedex 9, France}
\email{francois.chapon@math.univ-toulouse.fr}

\author{\' Eric Fusy}
\address{CNRS \and LIX, UMR CNRS 7161, \'Ecole  Polytechnique, 1 rue  Honor\'e  d'Estienne  d'Orves, 91120  Palaiseau, France}
\email{fusy@lix.polytechnique.fr}

\author{Kilian Raschel}
\address{CNRS \and Institut Denis Poisson, UMR CNRS 7013, Universit\'e de Tours et Universit\'e d'Orl\'eans, Parc de Grandmont, 37200 Tours, France}
\email{raschel@math.cnrs.fr}

\thanks{This project has received funding from the European Research Council (ERC) under the European Union's Horizon 2020 research and innovation programme under the Grant Agreement No.\ 759702.
F.C.\ is supported by the grant ANR-18-CE40-0006 MESA funded by the French National Research Agency (ANR). \'E.F.\ is supported by the grant ANR-16-CE40-0009-01 GATO funded by the French National Research Agency (ANR)}

\date{\today}

\begin{document}

\begin{abstract}
We investigate polyharmonic functions associated to Brownian motions and random walks in cones. These are functions which cancel some power of the usual Laplacian in the continuous setting and of the discrete Laplacian in the discrete setting. We show that polyharmonic functions naturally appear while considering asymptotic expansions of the heat kernel in the Brownian case and in lattice walk enumeration problems. We provide a method to construct general polyharmonic functions through Laplace transforms and generating functions in the continuous and discrete cases, respectively. This is done by using a functional equation approach.
\end{abstract}
\renewcommand{\subjclassname}{%
  \textup{2010} Mathematics Subject Classification}


\maketitle

\section{Introduction and motivations}

In the continuous setting, polyharmonic functions are functions which cancel some power of the usual Laplacian. More precisely,  a function $v$ on some domain $K$ of $\mathbb R^d$ satisfying
\begin{equation*}
     \Delta^p v = 0
\end{equation*}
for some $p\geq1$, where $\Delta$ is the usual Laplacian in $\mathbb R^d$, is said to be \textit{polyharmonic} of order $p$, or \textit{polyharmonic} for short. So polyharmonic functions of order 1 are just harmonic functions. 
Obviously, a polyharmonic function $v_p$ of order $p$ satisfies $\Delta v_p=v_{p-1}$, where $v_{p-1}$ is polyharmonic of order $p-1$. For example, polynomials are polyharmonic.
Harmonic functions have been tremendously investigated and pioneer works on polyharmonic functions go back to the work of Almansi \cite{Al-1899}. One can consult for instance  the monograph \cite{ArNaCr-83} for an introduction to this topic. 

In particular,  Almansi \cite{Al-1899} proved that if the domain $K$ is star-like with respect to the origin, then every polyharmonic function of order $p$ admits a unique decomposition
\begin{equation}
\label{eq:Almansi}
     f(x)=\sum_{k=0}^{p-1} \vert x\vert^{2k} h_k(x) ,
\end{equation}
where each $h_k$ is harmonic on $K$ and $\vert x\vert$ is the Euclidean length of $x$, hence completely characterising continuous polyharmonic functions on such domains. 

In comparison with the continuous case, much less is known in the discrete setting, where the Laplacian has to be replaced by a discrete difference operator. Some progress in understanding discrete polyharmonic functions has been made in the last two decades. For instance, one may cite \cite{CoCoGoSi-02}, where the authors investigated polyharmonic functions for the Laplacian on trees, and proved a similar result as Almansi's theorem \eqref{eq:Almansi} for homogeneous trees.  Recent works of Woess and co-authors \cite{HiWo-19, SaWo-19}  are  generalising this previous work. 

Our original motivation to study discrete polyharmonic functions comes from   the following framework. Consider a walk in $\mathbb Z^d$ with step set $\mathcal S$ confined in some cone $K\subset\mathbb Z^d$.  Denote by 
$q(x,y;n)$
the number of $n$-length excursions between $x$ and $y$ staying in the cone $K$. 
To simplify, we only consider the case where $y$ is the origin, but all considerations below can be generalised to $y\not = 0$. 
In various cases \cite{DeWa-15}, the asymptotics of $q(x,0;n)$ as $n\to\infty$ is known to admit the form
\begin{equation}
\label{eq:asym_equiv}
     q(x,0;n) \sim v_0(x) \gamma^n n^{-\alpha_0},
\end{equation}
where $v_0(x)>0$ is a function depending only on $x$, $\gamma\in(0,\vert\mathcal S\vert]$ is the exponential growth, and $\alpha_0$ is the critical exponent.
It is easy to see that the function $v_0(x)$ in \eqref{eq:asym_equiv} defines a discrete harmonic function. Indeed, plugging \eqref{eq:asym_equiv} into the obvious recursive relation
\begin{equation}
\label{eq:discrete_heat_equation}
     q(x,0;n+1) = \sum_{s\in\mathcal S} q(x+s,0;n)\mathbf{1}_{\{x+s\in K\}},
\end{equation}
dividing by $\gamma^{n+1} n^{-\alpha_0}$ and letting $n\to\infty$, we obtain 
\begin{equation}
\label{eq:C(x)_harmonic}
     v_0(x) = \frac{1}{\gamma} \sum_{s\in\mathcal S} v_0(x+s)\mathbf{1}_{\{x+s\in K\}},
\end{equation}
which proves that, with the assumption that $v_0(x)=0$ for $x\notin K$, $v_0(x)$ is discrete harmonic for the Laplacian operator
\begin{equation}
\label{eq:uniform_Laplacian}
     Lf(x)=\frac{1}{\gamma} \sum_{s\in\mathcal S} f(x+s)-f(x),
\end{equation}
that is, $Lv_0=0$.  Denisov and Wachtel \cite{DeWa-15} go further and show that
\begin{itemize}
     \item the exponential growth $\gamma$ is $\min_{\mathbb R_+^d} \sum_{(s_1,\ldots,s_d)\in\mathcal S}x_1^{s_1}\cdots x_d^{s_d}$, it does not depend on $K$;
     \item the critical exponent $\alpha_0$ equals $1+\sqrt{\lambda_1+(d/2-1)^2}$, where $d$ is the dimension and $\lambda_1$ is the principal Dirichlet eigenvalue on some spherical domain constructed from $K$.
\end{itemize}

As a leading example, consider the simple random walk in the quarter plane, with step set $\{\leftarrow, \uparrow, \rightarrow, \downarrow\}$. In this case, the number of excursions $q((i,j),0;n)$ is $0$ if $m=\frac{n-i-j}{2}$ is not a non-negative integer, and otherwise takes the value
\begin{equation}
\label{eq:SRW_excursion}
     q((i,j),0;n) =\frac{(i+1)(j+1)n!(n+2)!}{m!(m+i+j+2)!(m+i+1)!(m+j+1)!},
\end{equation}
see \cite{bou-02counting} and our Example \ref{ex:SRW}. The equivalence \eqref{eq:asym_equiv} is then
\begin{equation}
\label{eq:asymp_SRW}
     q((i,j),0;n) \sim \frac{4}{\pi} 4^n \frac{v_0(i,j)}{n^{3}},
\end{equation} where $v_0(i,j)=(i+1)(j+1)$ is the well-known unique (up to multiplicative constants)  harmonic function positive within the quarter plane with Dirichlet boundary conditions. Other examples of such asymptotics may be found for instance in \cite{BaFl-02,BMMi-10,CoMeMiRa-17}.

Our aim in this discrete setting is to study more precise estimates than \eqref{eq:asym_equiv}, by considering complete asymptotic expansions of the following form, as $n\to\infty$,
\begin{equation}
\label{eq:asym_full}
     q(x,0;n) \sim  \gamma^n \sum_{p\geq0} \frac{v_p(x)}{n^{\alpha_p}}.
\end{equation}
From such an asymptotic expansion and using similar ideas as in \eqref{eq:discrete_heat_equation}, \eqref{eq:C(x)_harmonic} and \eqref{eq:uniform_Laplacian}, it is rather easy to prove that the terms  $v_p$ are polyharmonic functions, in the sense that a power $L^kv_p$ of the Laplacian operator vanishes. We will provide examples of such asymptotic expansions (at least for the first terms)  and of  the set of exponents $\{\alpha_p\}_{p\geq0}$ appearing in~\eqref{eq:asym_full}.

On the other hand, the functional equation approach has proved to be fruitful when studying random walk problems. The reference book on this topic is the monograph \cite{FaIaMa-17} by Fayolle, Iasnogorodski and  Malyshev. This method has been used in \cite{Ra-14} to construct harmonic functions, both in the discrete and continuous settings. Basically, the method consists of drawing from the harmonicity condition a functional equation satisfied by the generating function   (in the discrete setting) or by the Laplace transform (in the continuous setting) of a harmonic function. Solving some boundary value problem for these quantities leads, via Cauchy or Laplace inversion, to the sought harmonic function. We will provide an implementation of this method to construct bi-harmonic functions, which can be generalised to polyharmonic functions. 

The main features of our results are as follows:
\begin{itemize}
     \item We shine a light on a new link between discrete polyharmonic functions and complete asymptotic expansions in the enumeration of walks. 
     \item Our approach provides tools to study complete asymptotics expansions as in \eqref{eq:asym_full}, but does not allow to prove their existence. On the other hand,
    the powerful approach of Denisov and Wachtel \cite{DeWa-15} seems restricted to the first term in the asymptotics \eqref{eq:asym_equiv}. Indeed, one of the main tools in \cite{DeWa-15} is a coupling result of random walks by Brownian motion, which only provides an approximation of polynomial order, see \cite[Lem.~17]{DeWa-15}.
     
     \item We introduce a new class of functional equations (see \eqref{functional-eq-BM-2} and \eqref{eq:functional_equation_2-harmo}), for which the method of Tutte's invariants introduced in \cite{Tu-95,BeBM-11,BeBMRa-17} proves to be useful.
     \item In the unweighted planar case, it has been shown \cite{BoRaSa-14} that knowing the rationality of the exponent $\alpha_0$ in \eqref{eq:asym_full} was sufficient to decide the non-D-finiteness of the series of excursions. However, for walks with big steps in dimension two or walk models in dimension three, this information is not enough \cite{BoBMMe-18}. As a potential application of our results, we might use arithmetic information on the other exponents $\alpha_p$ to study the algebraic nature, for example the transcendance, of the associated combinatorial series.
\end{itemize}


This paper is organised as follows. We choose to start with the continuous setting since computations are more enlightening and accessible. In Section~\ref{section-BM}, we prove that polyharmonic functions naturally arise when performing an asymptotic expansion of the Dirichlet heat kernel in a cone. We next present the functional equation method to construct polyharmonic functions. 
Our main result here is Theorem~\ref{thm:Laplace-2}, where a class of solutions for the Laplace transform of a bi-harmonic function  is provided. It shows that the Laplace transform of a bi-harmonic function can be expressed in terms of the Laplace transform of the related harmonic function plus some additional terms. This can be thought of as a Laplace transform version of Almansi's theorem \eqref{eq:Almansi}.
In Section~\ref{section-discrete}, we exhibit the same phenomenon in the random walk setting. Discrete polyharmonic functions appear when considering the asymptotic expansion of coefficients counting walks with fixed endpoints in a domain, and the functional equation approach may be used to construct discrete polyharmonic functions.

These notes are the starting point of a long-term research project on discrete polyharmonic functions in cones. Notice that many ideas and techniques are not specific to cones and would work for many other domains of restriction $K$.

\medskip

{\bf Acknowledgements.}
We would like to thank C\'edric Lecouvey, Steve Melczer, Pierre Tarrago and Wolfgang Woess for very interesting discussions. This project has started in July 2019, when two authors were invited at the Institute of Mathematical Statistics of M\"unster University. The institute, and in particular Gerold Alsmeyer, is greatly acknowledged for hospitality. The first author also acknowledges the Institut Denis Poisson for the warm hospitality during his stay at the Universit\'e de Tours, where part of this work has been pursued. Finally, we thank the three anonymous referees for useful comments.

\section{Classical polyharmonic functions and heat kernel expansions}
\label{section-BM}

As pointed out in \cite[Chap.~VI]{ArNaCr-83}, the connection between the heat kernel and polyharmonic functions is very profound. Here, we deepen this connection by proving an exact asymptotic expansion for the heat kernel in terms of polyharmonic functions. We then implement the functional equation method to construct polyharmonic functions.

\subsection{Exact asymptotic expansion for the Brownian semigroup in a cone}
\label{section:asympt-BM}

Let $K$ be some cone in $\mathbb R^d$ and consider the Brownian motion $(B_t)_{t\geq0}$ killed at the boundary of $K$. Denote by $p(x,y;t)$ its  transition density, that is the density probability function of the transition probability kernel
\begin{equation*}
     \mathbb P_x ( B_t \in dy,  \tau >t  ),
\end{equation*}
where $\tau$ is the first exit time of $K$. Recall the  well-known fact that $p(x,y;t)$ corresponds to the heat kernel, i.e., the fundamental solution of the heat equation on $K$ with Dirichlet boundary condition, see for instance \cite{BaSm-97}.
Here, we  prove that the heat kernel  admits a complete asymptotic expansion in terms of polyharmonic functions for the Laplacian. 

Denote by $\Delta$ the usual Laplacian on $\mathbb R^d$. In polar coordinates $(r,\theta)$, where $r$ is the radial part and $\theta$ the angular part, it writes:
\begin{equation}
\label{eq:Laplacian_B}
     \Delta = \frac{\partial^2}{\partial r^2} + \frac{d-1}{r}\frac{\partial}{\partial r} +\frac{1}{r^2}\Delta_{\mathbb S^{d-1}},
\end{equation}
where $\Delta_{\mathbb S^{d-1}}$ denotes the spherical Laplacian. Let respectively $m_j$ and $\lambda_j$ be the Dirichlet (normalised) eigenfunctions and eigenvalues for the spherical Laplacian on the generating set $K\cap\mathbb S^{d-1}$, that is,
\begin{equation}
\label{eq:eigenfunctions}
     \left\{\begin{array}{rcll}
     \Delta_{\mathbb S^{d-1}}m_j&=&-\lambda_j m_j & \text{in } K\cap\mathbb S^{d-1},\\
     m_j&=&0 & \text{in } \partial (K\cap\mathbb S^{d-1}).
     \end{array}\right.
\end{equation}
The eigenvalues satisfy $0<\lambda_1<\lambda_2\leq \lambda_3\leq \ldots$ by \cite[Chap.~VII]{Ch-84}.
We introduce, for $j\geq1$,
\begin{equation}
\label{eq:bj_betaj}
     \beta_j=\sqrt{\lambda_j+(d/2-1)^2}\quad \text{and}\quad b_j=1-d/2+\sqrt{\lambda_j+(d/2-1)^2}.
\end{equation}
Lemma~1 in \cite{BaSm-97} gives an explicit expression for the transition density $p(x,y;t)$ of the  Brownian motion in $K$. It states that, for $x,y\in \mathbb R^d$ and $t\in \mathbb R_+$, 
\begin{equation}
\label{eq:heat_kernel}
     p(x,y;t)=\frac{\exp\left(-\frac{\rho^2+r^2}{2t}\right)}{t(\rho r)^{\frac{d}{2}-1}}\sum_{j=1}^\infty I_{\beta_j}\left(\frac{\rho r}{t}\right) m_j(\theta)m_j(\eta),
\end{equation}
where in polar coordinates $x=(\rho, \theta)$ and $y=(r,\eta)$. Here,  $I_\beta$ is the modified Bessel function of the first kind of order $\beta$, satisfying the differential equation
$
     I''_\beta(z)+\frac{1}{z}I'_\beta(z)=(1+\frac{\beta^2}{z^2})I_\beta(z)
$
and admitting the series expansion
\begin{equation}
\label{eq:modified_Bessel_expansion}
     I_\beta(z) = \sum_{m=0}^{\infty}\frac{1}{m!\Gamma(m+\beta+1)}\left(\frac{z}{2}\right)^{2m+\beta}.
\end{equation}
The following easy lemma will allow us to define certain polyharmonic functions.
\begin{lemma}
\label{lem:deg-2}
For any $\mu\geq0$ and $j\geq 1$, let $f_{\mu,j}$ be defined in spherical coordinates by
\begin{equation}
\label{not:f_mu_j}
     f_{\mu,j}(r,\theta)=r^{\mu}m_j(\theta). 
\end{equation}
Then $f_{\mu,j}$ satisfies
\begin{equation}
\label{eq:f_mu_j}
     \Delta f_{\mu,j} =(\mu^2 +(d-2)\mu-\lambda_j) f_{\mu-2,j}.
\end{equation}
\end{lemma}
\begin{proof}
The proof is elementary using \eqref{eq:Laplacian_B} and \eqref{eq:eigenfunctions}.
\end{proof}
\begin{corollary}
\label{cor:poly_continuous}
For any $k\in\mathbb N$, the function $f_{b_j+2k,j}$ defined in~\eqref{not:f_mu_j} is $k$-polyharmonic.
\end{corollary}
\begin{proof}
It is obvious that $\mu=b_j$ satisfies $\mu^2 +(d-2)\mu-\lambda_j=0$, see \eqref{eq:bj_betaj}, so that $f_{b_j,j}$ is harmonic by \eqref{eq:f_mu_j}. An induction based on \eqref{eq:f_mu_j} completes the proof. 
\end{proof}
Doing an expansion of the heat kernel \eqref{eq:heat_kernel} as $t\to\infty$ and using series expansions of the exponential function and of the Bessel function \eqref{eq:modified_Bessel_expansion}, one immediately obtains:
\begin{theorem}
The Dirichlet heat kernel $p(x,y;t)$ in $K$ admits the following expansion, as $t\to \infty$, where $f_{b_j+2k,j}$ is defined in~\eqref{not:f_mu_j}, and $b_j$ and $\beta_j$ in~\eqref{eq:bj_betaj}:
\begin{multline*}
     p(x,y;t)\sim \\
     \sum_{j\geq1}\sum_{k,m\geq0 }\sum_{n=0}^{k}\frac{1}{t^{1+\beta_j+k+2m}} \frac{(-1)^k\binom{k}{n}}{2^kk!m!\Gamma(m+\beta_j+1)}f_{b_j+2(m+n),j}(\rho,\theta)f_{b_j+2(m+k-n),j}(r,\eta).
\end{multline*}
\end{theorem}

As such, the above result shows that the transition density of the Brownian motion in $K$ admits, as $t\to\infty$, an asymptotic expansion in descending powers of $t$ and in terms of polyharmonic functions for the Laplacian (see Corollary \ref{cor:poly_continuous}). Moreover, the set of these exponents  is (with $\mathbb N=\{0,1,2,\ldots\}$)
\begin{equation}
\label{eq:set_exponents}
     \bigcup_{j=1}^\infty (\beta_j+1+\mathbb N).
\end{equation}
Note that, depending on the cone, there might be an overlap between the sets $\beta_j+1+\mathbb N$. For instance, in the quadrant in dimension $2$, one has $\beta_j=2j$ and the set in \eqref{eq:set_exponents} reduces to $\{3,4,5,\ldots\}$. On the other hand, in dimension $2$ in a cone of opening $\alpha$ such that $\pi/\alpha\notin\mathbb Q$, there is no overlap between the points in \eqref{eq:set_exponents}.

As a last remark, we note that the same phenomenon appears for the survival probability $\mathbb P_x(\tau>t)$. Indeed, thanks to its explicit expression given by  \cite[Thm~1]{BaSm-97} (in terms of  the confluent hypergeometric function), one can write down an asymptotic expansion of $\mathbb P_x(\tau>t)$ in descending powers of $t$ in terms of polyharmonic functions for the Laplacian. 


\subsection{The functional equation approach}

We apply here the functional equation approach  in order to construct polyharmonic functions for the $2$-dimensional killed Brownian motion in a convex cone. This approach has been previously introduced in \cite{Ra-14} to compute harmonic functions, and is an adaptation of the functional equation method of the random walk case. Our main result is Theorem~\ref{thm:Laplace-2}, which gives the general form of the Laplace transform of a bi-harmonic function. 

Consider the Brownian motion $B$ in the quarter plane $\mathbb R^2_+$ (compared to the last section, we use $(x,y)$ for the coordinates of a 2d point) with covariance matrix
\begin{equation*}
     \Sigma = \begin{pmatrix} \sigma_{11} & \sigma_{12} \\ \sigma_{12} & \sigma_{22} \end{pmatrix} , 
\end{equation*}
with $\sigma_{11},\sigma_{22}>0$ and $\det{\Sigma} = \sigma_{11}\sigma_{22}- \sigma_{12}^2  \geq0$. Its 
infinitesimal generator is the operator
\begin{equation*}
\mathcal G f = \frac12 \left( \sigma_{11} \frac{\partial^2 f}{ \partial x^2}   +  2 \sigma_{12} \frac{\partial^2 f}{ \partial x \partial y} +  \sigma_{22} \frac{\partial^2 f}{ \partial y^2} \right) .
\end{equation*}
Note that through some linear transformation $\phi$ (see \cite[Eq.~(5.1)]{Ra-14}), one obtains the Brownian motion with identity covariance matrix in the cone $\phi(\mathbb R_+^2)$.

The  \textit{kernel} associated to the Brownian motion  is defined as the quantity 
\begin{equation*}
     \gamma (x,y) = \frac12 ( \sigma_{11} x^2  + 2  \sigma_{12} xy +  \sigma_{22} y^2), 
\end{equation*}
for $(x,y) \in \mathbb C^2$. The Laplace transform of a function $f$, which in the continuous case is the analogous quantity of the notion of generating function, is defined as
\begin{equation*}
L(f) (x,y) = \iint_{[0,\infty)^2} f(u,v) e^{-(xu+yv)} dudv,
\end{equation*}
for $(x,y) \in \mathbb C^2$ with positive real parts.

Now, let $h$ be a harmonic function associated with the Brownian motion with covariance matrix $\Sigma$, that is, $h$  vanishes on the boundary axes of the quadrant and satisfies
 $    \mathcal G h = 0$.
The functional equation for $h$ 
takes the following form (see \cite[Eq.~(A.1)]{Ra-14}):
\begin{equation*}
     \gamma(x,y) L(h)(x,y) = \frac12 ( \sigma_{11} L_1 (h)(y) + \sigma_{22} L_2 (h) (x)   ) +L(\mathcal G h)(x,y),
\end{equation*}
where we have denoted
\begin{equation*}
\left\{\begin{array}{lclcl}
L_1(h)(y) &:=&\displaystyle L \left( \frac{\partial h}{\partial x} (0,\cdot) \right) (y) &=& \displaystyle\int_0^\infty \frac{\partial h}{\partial x} (0,v) e^{-yv} dv ,\medskip\\
L_2(h)(x) &:=& \displaystyle L \left( \frac{\partial h}{\partial y} (\cdot,0) \right) (x) &= &\displaystyle \int_0^\infty \frac{\partial h}{\partial y} (u,0) e^{-xu} du.
\end{array}\right.
\end{equation*}
Using the harmonicity condition $\mathcal G h =0$, the functional equation for $h$ rewrites as
\begin{equation}
\label{functional-eq-BM-1}
     \gamma(x,y) L(h)(x,y) = \frac12 ( \sigma_{11} L_1 (h)(y) + \sigma_{22} L_2 (h) (x)   ) .
\end{equation}
We recall below the key argument of the method of \cite{Ra-14} to solve the functional equation~\eqref{functional-eq-BM-1}, which leads to harmonic functions for the Brownian motion via Laplace inversion. We will subsequently  apply a related method to obtain polyharmonic functions. 

Consider the two  solutions of 
$
\gamma \left(x,Y(x)\right)=0
$,
which, since $\gamma$ is a homogeneous polynomial of degree two,  are explicitly given by
$
Y_\pm(x) = c_\pm x
$,
with 
\begin{equation}
\label{eq:c_pm}
     c_\pm = \frac{-\sigma_{12}  \pm i \sqrt{\det{\Sigma}}}{\sigma_{22}}, 
\end{equation}
so that $c_+ = \overline{c_-}$. We write $c_\pm = c e^{\pm i\theta} $,  with $c = \sqrt{\frac{\sigma_{11}}{\sigma_{22}}}$ and $\theta$ such that $\cos \theta =- \frac{\sigma_{12}}{\sqrt{\sigma_{11}\sigma_{22}}}$. 

Denote by $\mathcal G_Y$ the domain delimited by the curve $Y_+([0,\infty])\cup Y_-([0,\infty])=c_+[0,\infty]\cup c_-[0,\infty]$ and containing the positive axis $[0,\infty]$.
Plugging each of the solutions $c_\pm x$ into the functional equation \eqref{functional-eq-BM-1}, one  obtains a \textit{boundary value problem} for $L_1(h)$, which states that:
\begin{enumerate}
\item $L_1(h)$ is analytic on $\mathcal G_Y$,
\item $L_1(h)$ is continuous  on $\overline{\mathcal G_Y}\setminus \{0\}$,
\item For all $x\in (0,\infty]$, $L_1(h)$ satisfies the boundary equation
$
     L_1(h)(c_+ x )  =  L_1 (h) (c_- x)
$.
\end{enumerate}

In order to solve this problem, one introduces  the conformal mapping $\omega$ from $\mathcal G_Y$ onto $\mathbb C \setminus \mathbb R_-$ defined by
$
\omega ( x) = {x^{-\pi/\theta}}
$.
 One eventually obtains that a class of solutions is obtained by letting  $L_1(h)$ to be of the form
\begin{equation}
\label{eq:L1-P}
     L_1(h)(y)= P\left(\frac{1}{y^{\pi/\theta}}\right),
\end{equation}
for any given polynomial $P$.
The same applies to $L_2(h)$ (by considering the solutions of $\gamma(X(y),y)=0$), and using the functional equation~\eqref{functional-eq-BM-1} and the fact that $(c_\pm)^{\pi/\theta}=-c^{\pi/\theta}$, one must have
\begin{equation*}
L_2(h)(y)= -\frac{\sigma_{11}}{\sigma_{22}} P\left(-\frac{1}{c^{\pi/\theta} x^{\pi/\theta}}\right),
\end{equation*}
with the  same $P$ as in \eqref{eq:L1-P}. Hence, using again the functional equation~\eqref{functional-eq-BM-1}, we deduce that  the Laplace transform of $h$ writes
\begin{equation}
\label{eq:LFP}
     L(h)(x,y) = \frac12 \sigma_{11} \frac{P\left(\frac{1}{y^{\pi/\theta}}\right)-P\left(-\frac{1}{c^{\pi/\theta} x^{\pi/\theta}}\right)}{\gamma(x,y)}.
\end{equation}
In particular, taking $P$ to be a polynomial of degree $1$, one gets 
\begin{equation*}
L(h)(x,y) =   \frac{ \sigma_{22}\frac{\mu_2}{x^{\pi/\theta}} +  \sigma_{11} \frac{\mu_1}{y^{\pi/\theta}}}{\gamma(x,y)} ,
\end{equation*}
where the constants are related by $\mu_2 = \mu_1 ( \frac{\sigma_{22}}{\sigma_{11}})^{1-\pi/2\theta}$. 
Taking the inverse Laplace transform, one should recover the unique positive harmonic function (written in polar coordinates $(\rho,\eta)$)
\begin{equation*}
h(x,y) = \rho^{\frac{\pi}{\theta}} \sin \left(  \frac{\pi}{\theta}   \eta  \right).
\end{equation*}

Suppose now that $v$ is bi-harmonic and satisfies
$
     \mathcal G v  = h
$,
where $h$ is harmonic. The functional equation for $v$ now reads
\begin{equation}
\label{functional-eq-BM-2}
\gamma(x,y) L(v)(x,y) = \frac12 ( \sigma_{11} L_1 (v)(y) + \sigma_{22} L_2 (v) (x)   ) + L(h)(x,y).
\end{equation}
By considering the roots of the kernel $\gamma$ and using the same method as above, we obtain
\begin{equation}
\label{eq:BVP-1-harmonic}
\frac12 \sigma_{11} L_1(v)(c_+ x) - \frac12 \sigma_{11} L_1(v) (c_- x) =  L(h)(x,c_- x ) - L(h)(x, c_+ x)  .
\end{equation}
We now have an \textit{a priori} non-homogeneous boundary  value problem for $v$, that we can in fact transform  into an homogeneous one, thanks to 
 the (already known) explicit form of $L(h)$. The key remark to this task is that $(c_+x)^{\pi/\theta} = (c_- x )^{\pi/\theta}=-(cx)^{\pi/\theta}$. Rewriting \eqref{eq:LFP} as
\begin{align*}
L(h)(x,y) 
& = \frac{\sigma_{11}}{\sigma_{22}} \frac{P\left(\frac{1}{y^{\pi/\theta}}\right)-P\left(\frac{1}{(c_\pm x)^{\pi/\theta}}\right)}{(y-c_- x)(y-c_+ x)}
\end{align*}
and letting $y\to c_+ x$ and $y\to c_- x$, one finds
\begin{equation*}
L(h)(x,c_\pm x) = \mp\frac{\sigma_{11}}{\sigma_{22}} \frac{\pi}{\theta}  \frac{1}{ (c_\pm x - c_\mp x)} 
P'\left(\frac{1}{(c_\pm x)^{\pi/\theta}}\right) \frac{1}{(c_\pm x)^{\pi/\theta+1}}.
\end{equation*}
Eventually, we get
\begin{align*}
& L(h)(x,c_- x ) - L(h)(x, c_+ x) \\
 & =
\frac{\sigma_{11}}{\sigma_{22}} \frac{\pi}{\theta} \left( 
 \frac{1}{ (c_+ x - c_- x)} \frac{P'\left(\frac{1}{(c_+ x)^{\pi/\theta}}\right)}{(c_+ x)^{\pi/\theta +1}}
 -
\frac{1}{(c_-  x- c_+ x)} \frac{P'\left(\frac{1}{(c_- x)^{\pi/\theta}}\right)}{(c_- x)^{\pi/\theta +1}} \right) \\
& = 
\frac{\sigma_{11}}{\sigma_{22}}  \frac{\pi}{\theta} \left( 
\frac{c_+}{c_+-c_-} P'\left(\frac{1}{(c_+ x)^{\pi/\theta}}\right) \frac{1}{(c_+ x)^{\pi/\theta +2} }
-
\frac{c_-}{c_- - c_+} P'\left(\frac{1}{(c_- x)^{\pi/\theta}}\right)\frac{1}{(c_- x)^{\pi/\theta +2} } \right) \\
& =
 - \frac{\sigma_{11}}{\sigma_{22}} \frac{\pi}{\theta} 
\frac{c_+c_-}{(c_+-c_-)^2} \left(  P'\left(\frac{1}{(c_+ x)^{\pi/\theta}}\right)   \frac{1}{(c_+ x)^{\pi/\theta +2} }
-
P'\left(\frac{1}{(c_- x)^{\pi/\theta}}\right) 
 \frac{1}{(c_- x)^{\pi/\theta +2} } \right),
\end{align*}
where the last equality follows from  $(c_+x)^{\pi/\theta} = (c_- x )^{\pi/\theta}$. Therefore, the boundary value equation~\eqref{eq:BVP-1-harmonic} is now homogeneous, and of the form
\begin{equation*}
\frac12 \sigma_{11} L_1(v)(c_+ x)-F(c_+ x ) = \frac12 \sigma_{11} L_1(v) (c_- x) -F(c_- x ),
\end{equation*}
where $F$ is equal on $Y_+([0,\infty])\cup Y_-([0,\infty])$ to 
\begin{equation}
\label{def:bigF}
F(y) =  -\frac{\sigma_{11}}{\sigma_{22}}  \frac{\pi}{\theta}  \frac{c_+c_-}{(c_+-c_-)^2} P'\left(\frac{1}{y^{\pi/\theta}}\right) \frac{1}{y^{\pi/\theta+2}} .
\end{equation}
We note that the simpler case when $F(c_+x)=F(c_- x)$ occurs exactly when $c_+^2=c_-^2$, i.e., $\theta$ is $0$ or $\pi/2$.
In this way, we obtain a boundary value problem analogous to the harmonic case, which, on the boundary of $\mathcal G_Y$ except at $0$, leads to 
\begin{equation*}
\frac12 \sigma_{11} L_1(v)(y)-F(y)=Q\left( \frac{1}{y^{\pi/\theta}} \right),
\end{equation*}
for any given polynomial $Q$.
The same computation applies  to $L_2(v)$. As such, using the equation~\eqref{functional-eq-BM-2}, the Laplace transform of the bi-harmonic function $v$ admits the following  form:
\begin{theorem}
\label{thm:Laplace-2}
For any polynomials $P$ and $Q$, the formula
\begin{equation*}
     L(v)(x,y) =  \frac{1}{\gamma(x,y)} \left[ Q\left( \frac{1}{y^{\pi/\theta}} \right) -Q \left( \frac{1}{(c_+x)^{\pi/\theta}} \right) + G(x,y) + L(h)(x,y) \right]
\end{equation*}
is the Laplace transform $L(v)$ of  a bi-harmonic function $v$ satisfying $\mathcal G v =h$, where $h$ is a harmonic function with Dirichlet boundary conditions, where the Laplace transform $L(h)$ of $h$ has the form \eqref{eq:LFP} and where
\begin{equation*}
G(x,y)= F(y)-F(c_+x)-L(h)(x,c_+x) ,
\end{equation*}
with $F$ defined in Eq.~\eqref{def:bigF}.
\end{theorem}
The above theorem can be understood as a Laplace transform counterpart of Almansi's theorem~\cite{Al-1899}.

Recursively, if $v_n$ is polyharmonic of order $n$ with $\mathcal G v_n= v_{n-1}$, where $v_{n-1}$ is polyharmonic of order $n-1$, the above method permits to express the Laplace transform of $v_n$ through the one of $v_{n-1}$, allowing to construct polyharmonic functions via Laplace inversion.

Further computations for the Brownian motion with identity covariance matrix are proposed in Appendix \ref{sec:app_BM}.

\section{Discrete polyharmonic functions}
\label{section-discrete}

Similarly to the continuous setting, we first investigate the appearance of polyharmonic functions in the asymptotic expansions of the counting coefficients of lattice paths with prescribed endpoints, starting from an exact expression for these coefficients (such exact expressions may typically be obtained from reflection principles). We then implement the functional equation approach to construct polyharmonic functions.

Our framework is thus the following. We consider random walks in the quarter plane $\mathbb Z^2_+$ with the following assumptions: 
\begin{enumerate}
\item The walk is homogeneous with transition probabilities $\{p_{i,j}\}_{-1\leq i,j \leq 1}$ to the eight nearest neighbours and $p_{0,0}=0$ (so we are only considering walks with small steps),
\item In the list $p_{1,1},p_{1,0},p_{1,-1},p_{0,-1},p_{-1,-1},p_{-1,0},p_{-1,1},p_{0,1}$, there are no three consecutive zeros (to avoid degenerate cases),
\item The drifts $\sum_{i,j}ip_{i,j}$ and $\sum_{i,j}jp_{i,j}$ are zero.
\end{enumerate}

The Markov operator $P$ of the walk is defined on discrete functions by
\begin{equation*}
Pf(x,y) = \sum_{-1\leq i,j \leq 1} p_{i,j} f(x+i,y+j),
\end{equation*}
and the Laplacian operator is $L=P-I$. A function $f$ is said to be \textit{harmonic} if $Lf=0$ and \textit{polyharmonic} of order $p$ if $L^pf=0$.  

\subsection{Examples of asymptotic expansion in walk enumeration problems}

We start by recalling a few exact expressions for the number of quarter plane walks of length $n$ with prescribed endpoints. 
\begin{example}[The diagonal walk]
The step set is $\{\nearrow,\nwarrow,\searrow,\swarrow  \}$, with uniform transition probabilities $\frac14$. It is well known (see for instance \cite{bou-02counting}) that
\begin{equation}
\label{counting-diagonal}
q((i,j),(0,0);n) = \frac{(i+1)(j+1)}{\frac{n+i+2}{2} \frac{n+j+2}{2}} \binom{n}{\frac{n+i}{2}} \binom{n}{\frac{n+j}{2}} ,
\end{equation}
with $i$ and $j$ having the same parity as $n$.
Starting from \eqref{counting-diagonal}, one can prove that 
\begin{equation}
\label{complete-asympt-diagonal}
 q((i,j),(0,0);n) \sim \frac{8}{\pi} 4^n  \sum_{p\geq0} \frac{v_p(i,j)}{n^{3+p}},
\end{equation}
where 
the first few terms in the above asymptotic expansion are given by
\begin{equation*}
\left\{\begin{array}{rcl}
v_0(i,j)& =& (i+1)(j+1) , \\
v_1(i,j) & =& -\frac12 (i+1)(j+1)(i^2+j^2+2i+2j+9) .
\end{array}\right.
\end{equation*}
The first term $v_0$ is the well-known unique (up to multiplicative constants) positive  harmonic function, with Dirichlet conditions; it is the same as for the simple walk, see \eqref{eq:asymp_SRW} and \eqref{eq:V3V4V5}. The next term satisfies $Lv_1=-3v_0$, and  therefore is  bi-harmonic. Note that in fact, using the explicit expression of the Laplacian $L$, it is obvious that any polynomial of degree at most $2p-1$ is polyharmonic of order $p$, since for any polynomial $f$ of degree $k$, $Lf$ has degree at most $k-2$ (it is a discrete equivalent of Lemma~\ref{lem:deg-2}).

To derive a full asymptotic expansion of \eqref{counting-diagonal}, we shall use the Laplace method applied to the counting coefficients rewritten as an integral, in the spirit of~\cite[p.~75--79]{spitzer} (alternatively one can apply the saddle-point method~\cite[Chap.\ B~VIII]{FlSe-09} in the framework of analytic combinatorics in several variables \cite{CoMeMiRa-17,MeWi-19}).  We choose to postpone it to Appendix \ref{sec:app}, since the computations are a bit long, though straightforward.
\end{example}

\begin{example}[The simple random walk]
\label{ex:SRW}
The step set is $\{\leftarrow, \uparrow, \rightarrow, \downarrow\}$, with uniform transition probabilities $\frac14$. We have \eqref{eq:SRW_excursion} by \cite{bou-02counting}.
Again, starting from \eqref{eq:SRW_excursion}, one can  prove that
\begin{equation*}
     q((i,j),(0,0);n) \sim \frac{4}{\pi}4^n\sum_{p\geq0} \frac{v_p(i,j)}{n^{3+p}},
\end{equation*}
where 
the first few terms in the asymptotic expansion are
\begin{equation}
\label{eq:V3V4V5}
     \left\{\begin{array}{rcl}
     v_0(i,j)&=&(i+1)(j+1),\\
     v_1(i,j)&=&-\frac{1}{4}(i+1)(j+1)(2i^2+2j^2+4i+4j+15). \\
     \end{array}\right.
\end{equation}
Again, $v_0$ is harmonic, and since  $Lv_1=-\frac{3}{2}v_0$, $v_1$ is  bi-harmonic.

\end{example}

\begin{example}[The tandem walk]
\label{ex:T}
The step set is $\{\nwarrow,\rightarrow,\downarrow\}$ with uniform transition probabilities $\frac13$. 
From \cite[Prop.~9]{BMMi-10}, we know that:
\begin{equation*}
     q((i,j),(0,0);n) = \frac{(i+1)(j+1)(i+j+2)(3m+2i+j)!}{m!(m+i+1)!(m+i+j+2)!},
\end{equation*}
with $n=3m+2i+j$. In this case, writing the asymptotic expansion
\begin{equation*}
q((i,j),(0,0);n)  \sim \frac{\sqrt 3}{2\pi} 3^n \sum_{p\geq0} \frac{v_p(i,j)}{n^{4+p}} ,
\end{equation*}
one has for the harmonic function $v_0$ and the bi-harmonic function $v_1$,
\begin{equation}
\label{eq:V3V4V5-tandem}
\left\{\begin{array}{rcl}
     v_0(i,j) & =& (i+1)(j+1)(i+j+2) , \\
     v_1(i,j)  & =&  -\frac19 (i+1)(j+1)(i+j+2) (3i^2  + 3j^2 +3ij+9i+9j+38) .
\end{array}\right.
\end{equation}
\end{example}

\subsection{Functional equation approach in the discrete case}

We implement here the functional equation method to construct polyharmonic functions. We start by recalling the key arguments in the harmonic case; details may be found in \cite{Ra-14}.

For a harmonic function $h$, we denote by $H$ its generating function, namely,
\begin{equation*}
H(x,y)= \sum_{i,j\geq0} h(i,j) x^i y^j. 
\end{equation*}
The \textit{kernel} of the random walk is defined as the polynomial
\begin{equation*}
     K(x,y)=xy\left(\sum_{-1\leq k, \ell\leq1 }p_{k,\ell}x^{-k}y^{-\ell}-1\right).
\end{equation*}
The harmonic equation $L h =0$ yields the following \textit{functional equation}
\begin{equation}
\label{eq:functional_equation_1-harmo}
     K(x,y)H(x,y)=K(x,0)H(x,0)+K(0,y)H(0,y)-K(0,0)H(0,0). 
\end{equation}

To solve \eqref{eq:functional_equation_1-harmo}, one first proves that the function $H(x,0)$ (and similarly $H(0,y)$) satisfies a \textit{boundary value problem} (see \cite{Ra-14}):
\begin{enumerate}
     \item $H(x,0)$ is analytic in $\mathcal G_X$,
     \item $H(x,0)$ is continuous  on $\overline{\mathcal G_X}\setminus \{1\}$,
     \item For all $x$ in the boundary of $\mathcal G_X$ except at $1$, $H(x,0)$ satisfies the boundary equation:
\begin{equation*}
     K(x,0)H(x,0)-K(\overline x ,0) H(\overline x ,0)=0.
\end{equation*}
\end{enumerate}
Here, $\mathcal G_X$ is a certain  domain bounded by the curve $X_+([y_1,1]) \cup X_-([y_1,1])$, where $X_\pm(y)$ are the branches of the algebraic function defined by $K(X(y),y)=0$.
Indeed, writing $K$ as 
\begin{equation*}
     K(x,y)=\widetilde\alpha(y) x^2 + \widetilde\beta(y) x + \widetilde\gamma(y),
\end{equation*}
where $\widetilde\alpha, \widetilde\beta,  \widetilde\gamma$ are polynomials of degree 2 whose coefficients depend on the model, we have
\begin{equation*}
     X_\pm (y) = \frac{-\widetilde\beta(y)\pm \sqrt{\widetilde\delta(y)}}{2\widetilde\alpha(y)},
\end{equation*}
where $\widetilde\delta(y) = \widetilde\beta(y)^2-4\widetilde\alpha(y)\widetilde\gamma(y)$. The functions $X_\pm$ are thus meromorphic on a cut plane, determined by the zeros of $\widetilde\delta$. 

It follows by \cite{Ra-14} that $K(x,0)H(x,0)$ may be written as a function of a certain conformal mapping $\omega$ (see \cite[Eq.~(3.1)]{Ra-14} for its explicit expression):
\begin{equation*}
     K(x,0)H(x,0) = P(\omega(x)),
\end{equation*}
where $P$ is an arbitrary entire function, for example a polynomial. This represents the analogous statement as \eqref{eq:L1-P} in the continuous setting. By the functional equation~\eqref{eq:functional_equation_1-harmo}, one eventually finds that
\begin{equation*}
     H(x,y) = \frac{P(\omega(x)) - P(\omega(X_+(x)))}{K(x,y)},
\end{equation*}
which again should be compared with \eqref{eq:LFP} in the continuous case.

For a bi-harmonic function $v$, satisfying $Lv=h$ with $h$ a harmonic function, the functional equation now writes
\begin{equation}
\label{eq:functional_equation_2-harmo}
     K(x,y)V(x,y)=K(x,0)V(x,0)+K(0,y)V(0,y)-K(0,0)V(0,0)-xyH(x,y),
\end{equation}
where $V$ is the  generating function of $v$,  i.e., $V(x,y)=\sum_{i,j\geq0} v(i,j) x^i y^j$; compare with \eqref{functional-eq-BM-2}. Notice that the equation \eqref{eq:functional_equation_2-harmo} is very close to functional equations coming up in walk enumeration problems.

Plugging the roots of the kernel into \eqref{eq:functional_equation_2-harmo}, one has
\begin{equation*}
     K(X_\pm(y),0)V(X_\pm(y),0)+K(0,y)V(0,y)-K(0,0)V(0,0)-X_\pm(y)yH(X_\pm(y),y)=0,
\end{equation*}
which leads to the boundary equation 
\begin{equation}
\label{eq:functional_equation_2-harmo_BVP}
     K(x,0)V(x,0)-K(\overline{x},0)V(\overline{x},0)= y\left( xH(x,y)-\overline x H(\overline x,y)    \right) , 
\end{equation}
for $x$ on the boundary of $\mathcal G_X$ (except at $1$). 

Note that a general method to solve this kind of boundary value problem \eqref{eq:functional_equation_2-harmo_BVP} exists \cite{FaIaMa-17}, for any quantity in the right-hand side, ending up in some contour integral expression for the unknown  function $K(x,0)V(x,0)$. We choose to provide below examples with simpler, integral-free expressions. Indeed, the resolution of \eqref{eq:functional_equation_2-harmo_BVP} is made easier in some peculiar cases, for instance when the right-hand side of  \eqref{eq:functional_equation_2-harmo_BVP} is zero (which occurs for the simple random walk, see Example~\ref{ex:SRW-funct-method} below), or when it can be decoupled in the terminology of \cite{BeBMRa-17} (which is analogous to the continuous setting and holds for the tandem walk, see Appendix~\ref{sec:app_tandem}).

\setcounter{example}{1}
\begin{example}[continued]
\label{ex:SRW-funct-method}
We consider here the case of the simple random walk,with kernel
\begin{equation*}
     K(x,y)=xy\left(\frac{1}{4}\left(x+\frac{1}{x}+y+\frac{1}{y}\right)-1\right).
\end{equation*}

The domain $\mathcal G_X$ is the open unit disk, and the conformal mapping $\omega$ admits the expression $\omega(x)=\frac{x}{(1-x)^2}$, see \cite{Ra-14}. A computation shows that $\omega(X_+(y))=-\omega(y)$, thus one gets that  the generating function of a harmonic function $h$ may be written as
\begin{equation*}
     H(x,y) = \frac{P(\omega(x))-P(-\omega(y))}{K(x,y)} . 
\end{equation*}
Choosing $P(x)=\frac{x}{4}$ leads to 
\begin{equation*}
     H(x,y)=\frac{\frac{\frac14x}{(1-x)^2}+\frac{\frac14y}{(1-y)^2}}{xy\left(\frac{1}{4}(x+\frac{1}{x}+y+\frac{1}{y})-1\right)}=\frac{1}{(1-x)^2(1-y)^2}=\sum_{i,j\geq0} (i+1)(j+1)x^{i}y^{j},
\end{equation*}
that is, $H$ is the generating function  of the unique positive harmonic function, see \eqref{eq:V3V4V5}.

We now consider bi-harmonic functions. Using the explicit form of $H$, one sees that the right-hand side of Eq.~\eqref{eq:functional_equation_2-harmo_BVP} vanishes. Indeed, we have
\begin{multline*}
X_+(y) H(X_+(y),y ) - X_-(y) H(X_-(y),y ) \\
= X_+(y)\frac{P'(\omega(X_+(y))) \omega'(X_+(y))}{\widetilde\alpha(y) (X_+(y)-X_-(y))} 
-X_-(y)\frac{P'(\omega(X_-(y))) \omega'(X_-(y))}{\widetilde\alpha(y) (X_-(y)-X_+(y))} ,
\end{multline*}
which is equal to zero since $\omega(X_+(y))=\omega(X_-(y))$ and
\begin{equation*}
X_+(y)\frac{\omega'(X_+(y))}{X_+(y)-X_-(y)} 
-X_-(y)\frac{ \omega'(X_-(y))}{X_-(y)-X_+(y)}=0
\end{equation*}
by straightforward computations.
The boundary equation has thus exactly the same form as the one in the harmonic case, so we get that on the boundary of $\mathcal G_X$,
\begin{equation*}
K(x,0) V(x,0) = Q(\omega(x)),
\end{equation*}
for some polynomial $Q$. Using (twice) the functional equation~\eqref{eq:functional_equation_2-harmo}, the general form for the generating function of a bi-harmonic $v$ satisfying $Lv=h$, with $h$ harmonic, is thus
\begin{equation*}
V(x,y)=\frac{Q(\omega(x))-Q(-\omega(y))+X_+(y)yH(X_+(y),y)-xyH(x,y)}{K(x,y)},
\end{equation*}
with 
\begin{equation*}
H(x,y) = \frac{P(\omega(x))-P(-\omega(y))}{K(x,y)} \quad \text{and} \quad H(X_+(y),y)= \frac{P'(\omega(X_+(y))) \omega'(X_+(y))}{\widetilde\alpha(y) (X_+(y)-X_-(y))}.
\end{equation*}
For instance, taking $P(x)=x$ and $Q$ the zero polynomial leads to the bi-harmonic function (non symmetrical in $i$ and $j$)
\begin{equation*}
v(i,j)=(i+1)j(j+1)  (j+2).
\end{equation*}
Indeed, one has
\begin{equation*}
X_+(y)H(X_+(y),y)=-\frac{y}{(1-y)^4},
\end{equation*}
so the generating function $V$ writes
\begin{equation*}
V(x,y)=\frac{-4y}{(1-x)^2(1-y)^4},
\end{equation*}
which is easily inverted. 
On the other hand, taking $P(x)=x$ and $Q(x)=-2x^2-\frac{5}{2}x$, one obtains the bi-harmonic function
\begin{equation*}
v(i,j)= (i+1)(j+1)(2i^2+2j^2+4i+4j+15),
\end{equation*}
which is (up to a multiplicative constant) the bi-harmonic function $v_1$ appearing  in 
Eq.~\eqref{eq:V3V4V5}. Another example will be treated in Appendix \ref{sec:app_tandem}.
\end{example}



\begin{thebibliography}{10}

\bibitem{Al-1899}
E.~{Almansi}.
\newblock {Sull'integrazione dell'equazione differenziale
  \(\varDelta^{2n}=0\).}
\newblock {\em {Annali di Mat. (3)}}, 2:1--51, 1899.

\bibitem{ArNaCr-83}
Nachman Aronszajn, Thomas~M. Creese, and Leonard~J. Lipkin.
\newblock {\em Polyharmonic functions}.
\newblock Oxford Mathematical Monographs. The Clarendon Press, Oxford
  University Press, New York, 1983.
\newblock Notes taken by Eberhard Gerlach, Oxford Science Publications.

\bibitem{BaSm-97}
Rodrigo Ba\~{n}uelos and Robert~G. Smits.
\newblock Brownian motion in cones.
\newblock {\em Probab. Theory Related Fields}, 108(3):299--319, 1997.

\bibitem{BaFl-02}
Cyril Banderier and Philippe Flajolet.
\newblock Basic analytic combinatorics of directed lattice paths.
\newblock {\em Theoretical Computer Science}, 281(1-2):37--80, 2002.

\bibitem{BeBM-11}
Olivier Bernardi and Mireille Bousquet-M\'{e}lou.
\newblock Counting colored planar maps: algebraicity results.
\newblock {\em J. Combin. Theory Ser. B}, 101(5):315--377, 2011.

\bibitem{BeBMRa-17}
Olivier Bernardi, Mireille Bousquet-M{\'e}lou, and Kilian Raschel.
\newblock Counting quadrant walks via {T}utte's invariant method.
\newblock {\em Preprint arXiv:1708.08215}, 2017.

\bibitem{BoBMMe-18}
Alin Bostan, Mireille Bousquet-M{\'e}lou, and Stephen Melczer.
\newblock Counting walks with large steps in an orthant.
\newblock {\em Preprint arXiv:1806.00968}, 2018.

\bibitem{BoRaSa-14}
Alin Bostan, Kilian Raschel, and Bruno Salvy.
\newblock Non-{D}-finite excursions in the quarter plane.
\newblock {\em J. Combin. Theory Ser. A}, 121:45--63, 2014.

\bibitem{bou-02counting}
Mireille Bousquet-M\'{e}lou.
\newblock Counting walks in the quarter plane.
\newblock In {\em Mathematics and computer science, {II} ({V}ersailles, 2002)},
  Trends Math., pages 49--67. Birkh\"{a}user, Basel, 2002.

\bibitem{BMMi-10}
Mireille Bousquet-M\'{e}lou and Marni Mishna.
\newblock Walks with small steps in the quarter plane.
\newblock In {\em Algorithmic probability and combinatorics}, volume 520 of
  {\em Contemp. Math.}, pages 1--39. Amer. Math. Soc., Providence, RI, 2010.

\bibitem{Ch-84}
Isaac Chavel.
\newblock {\em Eigenvalues in {R}iemannian geometry}, volume 115 of {\em Pure
  and Applied Mathematics}.
\newblock Academic Press, Inc., Orlando, FL, 1984.
\newblock Including a chapter by Burton Randol, With an appendix by Jozef
  Dodziuk.

\bibitem{CoCoGoSi-02}
Joel~M. Cohen, Flavia Colonna, Kohur Gowrisankaran, and David Singman.
\newblock Polyharmonic functions on trees.
\newblock {\em Amer. J. Math.}, 124(5):999--1043, 2002.

\bibitem{Co-12}
Louis {Comtet}.
\newblock {Advanced combinatorics. The art of finite and infinite expansions.}
\newblock {Dordrecht, Holland - Boston, U.S.A.: D. Reidel Publishing Company.
  X, 343 p. Dfl. 65.00.}, 1974.

\bibitem{CoMeMiRa-17}
J.~Courtiel, S.~Melczer, M.~Mishna, and K.~Raschel.
\newblock Weighted lattice walks and universality classes.
\newblock {\em J. Combin. Theory Ser. A}, 152:255--302, 2017.

\bibitem{DeWa-15}
Denis Denisov and Vitali Wachtel.
\newblock Random walks in cones.
\newblock {\em Ann. Probab.}, 43(3):992--1044, 2015.

\bibitem{FaIaMa-17}
Guy Fayolle, Roudolf Iasnogorodski, and Vadim Malyshev.
\newblock {\em Random walks in the quarter plane}, volume~40 of {\em
  Probability Theory and Stochastic Modelling}.
\newblock Springer, Cham, second edition, 2017.
\newblock Algebraic methods, boundary value problems, applications to queueing
  systems and analytic combinatorics.

\bibitem{FlSe-09}
Philippe Flajolet and Robert Sedgewick.
\newblock {\em Analytic combinatorics}.
\newblock Cambridge University Press, Cambridge, 2009.

\bibitem{HiWo-19}
Thomas Hirschler and Wolfgang Woess.
\newblock Polyharmonic functions for finite graphs and {M}arkov chains.
\newblock {\em Preprint arXiv:1901.08376}, 2019.
\newblock Frontiers in Analysis and Probability: in the Spirit of the
  Strasbourg-Zürich Meetings, Springer (to appear).

\bibitem{MeWi-19}
Stephen Melczer and Mark~C. Wilson.
\newblock Higher dimensional lattice walks: connecting combinatorial and
  analytic behavior.
\newblock {\em SIAM J. Discrete Math.}, 33(4):2140--2174, 2019.

\bibitem{Ra-14}
Kilian Raschel.
\newblock Random walks in the quarter plane, discrete harmonic functions and
  conformal mappings.
\newblock {\em Stochastic Process. Appl.}, 124(10):3147--3178, 2014.
\newblock With an appendix by Sandro Franceschi.

\bibitem{SaWo-19}
Ecaterina Sava-Huss and Wolfgang Woess.
\newblock Boundary behaviour of $\lambda$-polyharmonic functions on regular
  trees.
\newblock {\em Preprint arXiv:1904.10290}, 2019.

\bibitem{spitzer}
F.~Spitzer.
\newblock {\em Principles of {R}andom {W}alk}.
\newblock 2nd ed, Springer, New York, 1976.

\bibitem{Tu-95}
W.~T. Tutte.
\newblock Chromatic sums revisited.
\newblock {\em Aequationes Math.}, 50(1-2):95--134, 1995.

\end{thebibliography}

\newpage


\appendix

\section{Detailed computations for the standard Brownian motion in the quadrant}
\label{sec:app_BM}

Here we apply  the functional equation  approach to the case of the Brownian motion in the quarter plane with identity covariance matrix. The kernel $\gamma$ is equal to 
$
\gamma(x,y)=\frac12\left(x^2+y^2\right),
$
so $c_\pm = \pm i$ and $\theta=\frac{\pi}{2}$, see \eqref{eq:c_pm}. The functional equation~\eqref{functional-eq-BM-1} for $h$ harmonic is then
\begin{equation*}
(x^2+y^2) L(h)(x,y) = L_1(h)(y) + L_2(h)(x) ,
\end{equation*} which leads to
\begin{equation*}
L(h)(x,y) =  \frac{P\bigl(\frac{1}{y^2}\bigr) - P\bigl(-\frac{1}{x^2}\bigr)}{x^2+y^2}. 
\end{equation*}
In case when $P$ is the degree $1$ polynomial $P(x)=x$, one gets
$
L(h)(x,y)= \frac{1}{x^2y^2}
$
which is the Laplace transform of the well-known unique positive harmonic function within the quarter plane 
$
h(x,y) = xy .
$

More generally, the choice of  $P(x)=-(2j)!(-x)^j$ leads  to the Laplace transform (in Cartesian coordinates) of the harmonic function $f_{2j,j}$ defined in \eqref{not:f_mu_j}. Indeed, recall that
$f_{2j,j} (\rho,\theta) = \rho^{2j} \sin \left( 2j \theta \right)$,
which is written in Cartesian coordinates as follows. Recall that the Chebyshev polynomial $U_j$ of the second kind  is defined as
$
U_{j} (\cos \theta ) \sin \theta = \sin( j \theta)
$, $j\geq0$,
and admits the expression
\begin{equation*}
U_{j}(z)=z^{j} \sum_{k=0}^{\lfloor j/2 \rfloor}  \binom{j+1}{2k+1} (1-z^{-2})^k .
\end{equation*}
Hence, thanks to the explicit expression of $U_{2j-1}$, the harmonic function $f_{2j,j}$ can be written,
 in Cartesian coordinates $(x,y)=(\rho\cos\theta,\rho\sin\theta)$,
\begin{equation*}
f_{2j,j} (x,y) =  \sum_{k=0}^{j-1} (-1)^k \binom{2j}{2k+1} y^{2k+1} x^{2j-(2k+1)} .
\end{equation*}
The Laplace transform of $f_{2j,j}$ is now computed using $L(x^ny^k)=\frac{n!k!}{x^{n+1}y^{k+1}}$,  and one obtains
\begin{equation}
\label{eq:Laplace-f2j}
L(f_{2j,j}) (x,y) 
 = (2j)! \sum_{k=0}^{j-1}  (-1)^k \frac{1}{y^{2k+2} x^{2j-2k}} 
 = (2j)! \frac{ \bigl(\frac{1}{x^2}\bigr)^j  -  \bigl(-\frac{1}{y^2}\bigr)^j}{x^2+y^2} .
\end{equation}


For $v$ bi-harmonic, the functional equation~\eqref{functional-eq-BM-2} is
\begin{equation*}
(x^2+y^2) L(v)(x,y) = L_1(v)(y) + L_2(v)(x)  + 2L(h)(x,y) , 
\end{equation*}
and 
the general form of the Laplace transform of $v$ writes 
\begin{equation}
\label{eq:Laplace-g}
L(v)(x,y) =  \frac{Q\bigl(\frac{1}{y^{2}}\bigr)  - Q\bigl(-\frac{1}{x^{2}}\bigr) +\frac{2}{x^4}P'\bigl( -\frac{1}{x^2} \bigr)+  2\frac{P\bigl(\frac{1}{y^{2}}\bigr)  - P\bigl(-\frac{1}{x^{2}}\bigr)}{x^2+y^2}}{x^2+y^2}, 
\end{equation}
where $P$ and $Q$ are arbitrary polynomials. Choosing $P(x)$ equal to $x$
and $Q(x)$ of degree $2$, equal to  $x^2$, gives that
\begin{equation*}
L(v)(x,y) = \frac{x^2+y^2}{x^4y^4}= \frac{1}{x^2y^4} + \frac{1}{x^4y^2} , 
\end{equation*}
which is the Laplace transform of the function
$
v(x,y) = (x^2+y^2) xy
$,
which corresponds in polar coordinate $(\rho,\theta)$ to the bi-harmonic function
$
f_{4,2}(\rho,\theta)=\rho^{4} \sin{2\theta}
$
defined in \eqref{not:f_mu_j}.

More generally, choosing 
\begin{equation*}
     P(x)=(-1)^{j+1}(2j)!2(2j+1)x^j\quad \text{and}\quad Q(x)=(-1)^{j+1}(2j)!2(2j+1)jx^{j+1}
\end{equation*}     
leads to the bi-harmonic function $f_{2j+2,j}$. Indeed, since \sloppy 
$f_{2j+2,j}(x,y)=(x^2+y^2)f_{2j,j}(x,y)$, one has, from the usual properties of the Laplace transform, that
$
L(f_{2j+2,j}) = \Delta L(f_{2j,j})$. As such, by applying the Laplacian to the Laplace transform of $f_{2j,j}$ given in \eqref{eq:Laplace-f2j}, one obtains that
\begin{multline*}
  L(f_{2j+2,j})(x,y) =  \\
  \frac{(2j)!2(2j+1)}{x^2y^2(x^2+y^2)^2}  \left\{ (j+2) x^2y^2  \left( \Bigl(\frac{1}{x^2}\Bigr)^j - \Bigl(\frac{-1}{y^2}\Bigr)^j\right) 
 + j \left( y^4 \Bigl(\frac{1}{x^2}\Bigr)^j - x^4 \Bigl(\frac{-1}{y^2}\Bigr)^j  \right)  \right\} .
\end{multline*}
Now, plugging the above choice of $P$ and $Q$ in Eq.~\eqref{eq:Laplace-g} gives easily the  formula.

\section{Complete asymptotic expansion for the diagonal walk}
\label{sec:app}

As an explicit example, we provide a complete asymptotic expansion for the number \eqref{counting-diagonal} of $n$-excursions from the origin to $(i,j)$ for the diagonal walk with steps from $\{\nearrow,\nwarrow,\searrow,\swarrow\}$.
A straightforward way to obtain such an asymptotic expansion  is to apply the standard Laplace's method (see \cite[p.~755]{FlSe-09}) using an integral representation of  \eqref{counting-diagonal} (in~\cite[p.~75--79]{spitzer}, this is applied to obtain first order asymptotic estimates in lattice paths enumeration problems). This leads to  an explicit new family of polynomials $(v_p)_{p\geq 0}$ of increasing degree, where $v_p$  is  the polyharmonic function of order $p+1$ appearing in the expansion \eqref{complete-asympt-diagonal}, see Corollary~\ref{cor:complete-asymptotic}.

Let us first introduce the necessary notations.
Projecting the walk onto the coordinate axes, one gets two decoupled prefixes of Dyck paths. Hence  \eqref{counting-diagonal} is obtained by
a simple application of the reflection principle in the one-dimensional case, which gives that the number of non-negative paths from 0 to $\lambda$ with $n$ steps is given by
\begin{equation}
\label{def:counting-1D}
m(\lambda,n):=\binom{n}{\frac{n+\lambda}{2}} -  \binom{n}{\frac{n+\lambda+2}{2}} = \frac{\lambda+1}{\frac{n+\lambda+2}{2} } \binom{n}{\frac{n+\lambda}{2}} ,
\end{equation}
with $\lambda \equiv n\mod 2$.
Using the simple integral representation of the binomial coefficient
\begin{equation*}
\binom{n}{k} = \frac{1}{2\pi} \int_{-\pi}^{\pi} e^{-ikt} (1+e^{it})^n dt,
\end{equation*}
one readily obtains the following integral representation for $m(\lambda,n)$:
\begin{equation}
\label{eq:def_m}
m(\lambda,n) =  \frac{2}{\pi}  \int_{-\pi/2}^{\pi/2} 2^n (\cos y)^n \sin((\lambda+1)y) \sin(y) dy.
\end{equation}
Now define
the  sequence $(\alpha(m))_{m\geq1}$ as
\begin{equation}
\label{def:alphanumb}
\alpha(m)=  \frac{(4^m-1)\vert B_{2m}\vert2^{2m}}{2m(2m)!},
\end{equation}
where the $B_{2m}$'s are the Bernoulli numbers, which can be defined through the Riemann zeta function at even integers:
\begin{equation*}
\zeta(2m)=\frac{\vert B_{2m}\vert(2\pi)^{2m}}{2(2m)!}. 
\end{equation*} 
Define also, for $s\geq k \geq 0$,
\begin{equation}
\label{def:Bell}
B_{s,k}^\alpha  := B_{s,k} \left(\alpha(2), \ldots, \alpha(s-k+2) \right) ,
\end{equation}
the rational numbers obtained by evaluating the  partial ordinary Bell polynomial in the variables $\alpha(m+1)$. 
Recall that by definition, see for instance \cite{Co-12}, the partial ordinary Bell polynomials in the variables $(x_k)_{k\geq1}$ are the polynomials obtained by performing the formal double series expansion:
\begin{equation*}
\exp\Big( u \sum_{m\geq1} x_m t^m \Big) = \sum_{n\geq k \geq 0} B_{n,k}(x_1,\ldots,x_{n-k+1} ) t^n \frac{u^k}{k!} .
\end{equation*}
Note that the polynomial $B_{n,k}$ contains $p(n,k)$ monomials, where $p(n,k)$ stands for the number of partitions of $n$ into $k$ parts, see  \cite{Co-12} for details and for an explicit expression of these polynomials.
Finally, define for $p\geq k \geq0$,
\begin{equation}
\label{def:BB}
C^\alpha_{k,p}=\frac{1}{k!}\sum_{j=k}^p  \frac{(-1)^j}{(2p-2j+1)!} B_{j,k}^\alpha.
\end{equation}

We first give  a complete asymptotic expansion for prefixes of Dyck paths.
\begin{theorem}
\label{thm:poly-1D}
Let $m(\lambda,n)$ be the number of non-negative paths from $0$ to $\lambda\in \mathbb Z_+$ given by \eqref{def:counting-1D}. The following asymptotic expansion holds as $n\to\infty$:
\begin{equation*}
     m(\lambda,n) \sim 2 \sqrt 2 \frac{2^n }{\sqrt \pi} \frac{1}{n^{3/2}} 
\sum_{j\geq0} \frac{(-1)^j}{n^j} h_j(\lambda),
\end{equation*}
where  for $j\geq0$,
\begin{equation}
\label{def:poly-1D}
h_j(\lambda) 
 = \sum_{p=0}^j \sum_{k=0}^p \frac{ (-1)^k}{(2(j-p)+1)!} C^\alpha_{k,p} m_{2(k+j+1)} 
(\lambda+1)^{2(j-p)+1}  ,
\end{equation}
where $m_{2k}=\frac{(2k)!}{2^k k!}$ is the $2k$-th Gaussian moment and $C^\alpha_{k,p}$ is defined in~\eqref{def:BB}.
\end{theorem}
Hence, the above theorem gives, in the one-dimensional case, an asymptotic expansion of the number of non-negative paths in terms of polyharmonic functions. Indeed, it is easily seen that the polynomial $h_j$ has degree $2j+1$, so is polyharmonic of order $j+1$ for the one-dimensional Laplacian $Lf(x)=\frac1 2(f(x+1)+f(x-1))-f(x)$.

Since the number of $n$-excursions for the diagonal walk is the product of two numbers of (decoupled) Dyck paths, one readily obtains the following corollary.

\begin{corollary}
\label{cor:complete-asymptotic}
Let $q(0,(i,j);n)$ be the number of diagonal paths with $n$ steps from the origin to $(i,j)$ and confined in the quadrant, given by \eqref{counting-diagonal}. Then 
\begin{equation*}
q(0,(i,j);n) \sim  \frac{8}{\pi} \frac{1}{n^3} 4^n \sum_{p\geq0} \frac{(-1)^p}{n^p} v_p(i,j),
\end{equation*}
where, with $h_k$ defined in \eqref{def:poly-1D},
\begin{equation*}
v_p(i,j) = \sum_{k=0}^p h_k(i) h_{p-k}(j).
\end{equation*}
\end{corollary}
Clearly, the polynomial function $v_p$ has degree $2p+1$ and thus is polyharmonic of order $p+1$ for the Laplacian associated to the diagonal walk. The set of exponents \eqref{eq:set_exponents} appearing in the asymptotic expansion is here $3+\mathbb N$.

\begin{proof}[Proof of Theorem~\ref{thm:poly-1D}]
To obtain the claimed asymptotic expansion, we apply the Laplace method as in \cite[p.~755]{FlSe-09} to the integral representation of $m(\lambda,n)$ in \eqref{eq:def_m}.
Indeed, the cosine function admits only one maximum in the interval $[-\frac{\pi}{2},\frac{\pi}{2}]$, at $y=0$, and the contribution to the integral outside any fixed segment containing $0$ is exponentially small and as such can be discarded for an asymptotic consideration.

So, first, we perform the change of variable $\theta=\frac{y}{\sqrt n}$ to get
\begin{equation*}
 m(\lambda,n) 
 = 2^n \frac{ 2}{\pi}  \frac{1}{n^{1/2}} \int_{-\frac{\pi}{2}\sqrt n}^{\frac{\pi}{2}\sqrt n} \cos \left( \frac{y}{\sqrt n} \right)^n \sin  \left( \frac{y}{\sqrt n} \right)  \sin  \left( (\lambda+1)\frac{y}{\sqrt n} \right)dy  .
\end{equation*}
The next step is to consider an asymptotic expansion of the integrand as $n\to\infty$. Using the Weierstrass product formula for the cosine function,
\begin{equation*}
\cos  y  =\prod_{k=1}^\infty \left( 1-\frac{4y^2}{\pi^2(2k-1)^2}    \right) 
\end{equation*}
and the Taylor series of the logarithm function, one has
\begin{equation*}
\log \cos \left( y\right)  =  
- \sum_{m\geq1}\alpha(m) y^{2m} ,
\end{equation*}
where the sequence $(\alpha(m))_{m\geq1}$ is defined in \eqref{def:alphanumb}. Note that an interpretation of the sequence $(\alpha(m))_{m\geq1}$ is that they correspond to the cumulant sequence of the Bernoulli distribution $\frac12\delta_{+1} +\frac12\delta_{-1}$. 
Now one has, using $\alpha(1)=\frac12$ and the Taylor series of the exponential function,
\begin{equation*}
\cos \left( \frac{y}{\sqrt n} \right)^n 
 = \exp \left( n \log \cos \left( \frac{y}{\sqrt n} \right)  \right)
= e^{-y^2/2} \sum_{s\geq0} \frac{1}{n^s} \sum_{k=0}^s \frac{(-1)^k}{k!}  B_{s,k}^\alpha \ y^{2(k+s)} ,
\end{equation*}
where $B_{s,k}^\alpha$ is the partial ordinary Bell polynomial  defined in \eqref{def:Bell}. Now, using the Taylor series of the sine function, and after some elementary manipulations, one gets
\begin{multline*} 
 \cos \left( \frac{y}{\sqrt n} \right)^n \sin  \left( \frac{y}{\sqrt n} \right) \sin  \left( (\lambda+1)\frac{y}{\sqrt n} \right) \\
 = e^{-y^2/2}  \frac{1}{n} 
\sum_{j\geq0} \frac{(-1)^j}{n^j} \sum_{p=0}^j \sum_{k=0}^p (-1)^k \frac{C^\alpha_{k,p}}{(2(j-p)+1)!} y^{2(k+j)+2}  (\lambda+1)^{2(j-p)+1} ,
\end{multline*}
where $C^\alpha_{k,p}$ is defined in \eqref{def:BB}. 

The next step in the Laplace method is to neglect the tails. Hence, we write
\begin{equation*}
 m(\lambda,n) 
 \sim \frac{2 }{\pi}  \frac{2^n}{n^{3/2}} 
\sum_{j\geq 0} \frac{(-1)^j}{n^j} \sum_{p=0}^j \sum_{k=0}^p  \frac{(-1)^kC^\alpha_{k,p}}{(2(j-p)+1)!}  (\lambda+1)^{2(j-p)+1}
\int_{-\kappa_n}^{\kappa_n}  e^{-y^2/2} y^{2(k+j)+2} dy ,
\end{equation*}
where $\kappa_n$ is chosen so that the error bounds are exponentially small (for instance one can choose arbitrarily $\kappa_n=n^{1/10}$). Completing the tails of the Gaussian integral, that is
\begin{equation*}
\int_{-\kappa_n}^{\kappa_n}  e^{-y^2/2} y^{2(k+j)+2} dy \sim \int_{\mathbb R} e^{-y^2/2} y^{2(k+j)+2} dy 
 = \sqrt{2\pi}  \frac{(2(k+j+1))!}{2^{k+j+1}(k+j+1)!}  =\sqrt{2\pi}m_{2(k+j+1)},
\end{equation*}
where $m_{2k}=\frac{(2k)!}{2^k k!}$ is the $2k$-th Gaussian moment, one finally obtains, with $h_j$ defined in~\eqref{def:poly-1D}, that
 \begin{equation*}
m(\lambda,n) \sim 2 \sqrt 2 \frac{2^n }{\sqrt \pi} \frac{1}{n^{3/2}} 
\sum_{j\geq0} \frac{(-1)^j}{n^j} h_j(\lambda). \qedhere
\end{equation*}
\end{proof}

\section{The example of tandem walks}
\label{sec:app_tandem}

In this subsequent example, we consider the tandem walk with steps from $\{\nwarrow,\rightarrow,\downarrow\}$, see Example \ref{ex:T}. In this case,  the functional equation approach admits a nicer form because the right-hand side of 
Eq.~\eqref{eq:functional_equation_2-harmo_BVP} can be decoupled, that is, can be written as
$
G(X_+(y))-G(X_-(y)),
$
for some function $G$. The computations are close to the continuous case but are quite tedious.
 First, we know \cite{Ra-14} that the generating function $H$ of a harmonic function $h$ is of  the form
\begin{equation}
\label{eq:H-tandem}
H(x,y) = \frac{P(\omega(x))-P(\omega(X_+(y)))}{K(x,y)},
\end{equation}
where the conformal mapping $\omega$ is given by $\omega(x)=\frac{x^2}{(1-x)^3}$. The unique positive harmonic function $v_0(i,j)=\frac{1}{2}(i+1)(j+1)(i+j+2)$ of \eqref{eq:V3V4V5-tandem} is obtained  choosing $P(x)=\frac13 x$.

Using the general form of $H$, one has
\begin{multline*}
 yX_+(y)H(X_+(y),y)-yX_-(y)H(X_-(y),y)=  \\
 3\frac{yX_+(y)\omega'(X_+(y))}{X_+(y)-X_-(y)} P'(\omega(X_+(y))) 
- 
3\frac{yX_-(y)\omega'(X_-(y))}{X_-(y)-X_+(y)} P'(\omega(X_-(y))).
\end{multline*}
Define now the \textit{decoupling} function on $\mathcal G_X$:
\begin{equation}
\label{eq:Fdecoupled}
     F(x) =  -\frac{x^3}{(1-x)^6}.
\end{equation}
Some computations show that
\begin{equation*}
     \frac{yX_+(y)\omega'(X_+(y))}{X_+(y)-X_-(y)} 
     - 
     \frac{yX_-(y)\omega'(X_-(y))}{X_-(y)-X_+(y)}  = F(X_+(y))-F(X_-(y)).
\end{equation*}
A crucial point is to \textit{guess} the function $F$ in \eqref{eq:Fdecoupled} satisfying the above equation. Minding the fundamental fact that $\omega(X_+(y))=\omega(X_-(y))$, it follows that
\begin{equation*}
yX_+(y)H(X_+(y),y)-yX_-(y)H(X_-(y),y) = G(X_+(y))-G(X_-(y)),
\end{equation*}
where $G(x)=3F(x) P'(\omega(x))$. One deduces that the generating function $V(x,y)$ for a bi-harmonic function $v$ satisfying $Lv=h$ admits the form
\begin{equation*}
     \frac{1}{K(x,y)} \Big(  Q(\omega(x))-Q(\omega(X_+(y))) +G(x) - G(X_+(y))  
+ X_+(y)yH(X_+(y),y) - xy H(x,y)  \Big) ,
\end{equation*}
where $H$ has the general form given by Eq.~\eqref{eq:H-tandem} and $G(x)=3F(x) P'(\omega(x))$ with the decoupling function $F$ defined in Eq.~\eqref{eq:Fdecoupled}. Note that this has to be compared with Theorem~\ref{thm:Laplace-2}. 

Choosing $P(x)=x$ and $Q=0$ leads to the bi-harmonic function
\begin{equation}
\label{eq:biharmoQ0-tandem}
v(i,j)= (j+1)(i+1)(i+j+2)(2i^3+3i^2j+14i^2+5ij+24i-3ij^2-2j^3-4j^2+6j).
\end{equation}
To obtain to bi-harmonic function $v_1$ of \eqref{eq:V3V4V5-tandem}, one chooses  $P(x)=-\frac89 x$ and $Q(x)=\frac{8}{3}x^2+\frac{76}{27}x$. This is obtained by noticing that an appropriate linear combination 
of the bi-harmonic function \eqref{eq:biharmoQ0-tandem} and of $v_1$ is harmonic and its generating function corresponds to the term
\begin{equation*}
     \frac{Q(\omega(x))-Q(\omega(X_+(y)))}{K(x,y)}.
\end{equation*}
As such, computing its generating function leads to the polynomial $Q$.

\end{document}